\documentclass[12pt,leqno,fleqn]{amsart} 
\usepackage{amsmath}         
\usepackage{amsfonts} 
\usepackage{graphicx} 
\usepackage{mathrsfs}
\usepackage{amsrefs} 
\usepackage{float}
\usepackage{amsmath,amstext,amsthm,amssymb,amsxtra}
\usepackage[top=1.5in, bottom=1.5in, left=1.25in, right=1.25in]    {geometry}

\numberwithin{equation}{section}
\usepackage{tikz}
\usetikzlibrary{arrows}
\usepackage{mathtools}
\mathtoolsset{showonlyrefs,showmanualtags} 
\usepackage{amssymb}

\newtheorem{theorem}{Theorem}[section]
\newtheorem{lemma}[theorem]{Lemma}
\numberwithin{equation}{theorem}

\title{Phase Transition in the One-bit Johnson-Lindenstrauss Lemma} 
\author{Amadou Bah} 
\author{Bryson Kagy} 
\author{Emily Smith}  

\thanks{Research conducted at during an REU sponsored by an NSF MCTP-Grant to the Georgia Institute of Technology} 

\begin{document}
\begin{abstract}
The Johnson-Lindenstrauss Lemma (J-L Lemma) is a cornerstone of dimension reduction techniques.
 We study it in the one-bit context,  namely we consider the unit sphere $ \mathbb S ^{N-1}$, with 
 normalized geodesic metric,  and map a finite set  $ \mathbf{X} \subset \mathbb{S}^{N-1}$ into the Hamming cube $\mathbb{H}_m = \{0,1\}^m$, with normalized Hamming metric. 
We find that for $ 0< \delta <1$,  and $m>\frac{\ln n}{2\delta^2}$ there is a $\delta$-RIP  from $\mathbf{X}$ into $\mathbb{H}_m$.  This is surprising as the value of $ m$ is virtually identical to best known bound  linear J-L Lemma.  
In both the linear and one-bit case, the maps are randomly constructed.  
 We show that the probability of $B_m$  being a $\delta$-RIP  satisfies a phase transition. 
 It passes from probability of  nearly zero to nearly one with a very small change in $m$.
 Our proof relies on delicate properties of Bernoulli random variables.
\end{abstract}

\maketitle 
\section{Introduction}

Compressive sensing was first introduced as a practical application of signal processing and has since taken off and proven to be very useful for many aspects of modern life such as MRI scanning, cell phone imaging, electron microscopy, and many more \cite{lustig2007sparse, fornasier2011compressive,binev2012compressed}. It has been previously shown by Johnson-Lindenstrauss \cite{johnson1984extensions}, that given a very high dimensional data set in $\mathbb{R}^N$, it is possible, with little sacrifice, to map vectors from a subset of this $N$-dimensional space, to a much lower, $m$-dimensional space.   Recently, Alon and Klartag  \cite{2016arXiv161000239A} studied the minimum number of bits required in order to maintain the Euclidean distance between data points. This differs from our results through the fact that we maintain the geodesic distance between points.  The non-linear geodesic metric is basic to our considerations.

Dasgupta-Gupta \cite{dasgupta1999elementary} provide the best quantitative bounds in the J-L Lemma. 
  For any $0<  \delta <1$, any integer $ n$  let 
$$m\geq 4\frac{\ln n}{\frac{\delta ^2}{2}-\frac{\delta ^3}{3}},$$  
then for any set of $ n$  points in $\mathbb{R}^N$
there exist a map $f:\mathbb{R}^N\rightarrow\mathbb{R}^m$ such that for all $x, y \in \mathbb{R}^N$ we have:
$$(1-\delta)\left\|x-y\right\|^2\leq \left\|f(x)-f(y)\right\|^2\leq (1+\delta)\left\|x-y\right\|^2$$
For comparison below, we remark that $ \sqrt {1 \pm \delta } \simeq 1 \pm \delta /2$.   
 
We study the one bit context.  Consider the unit sphere $ \mathbb S ^{N-1}$ with the normalized geodesic metric. We map finite $ \mathbf X \subset \mathbb S ^{N-1}$ into the $ m$ dimensional   Hamming cube $ \mathbb H _m$ , with normalized 
Hamming metric.  A main result is that for $ 0< \delta < 1$, and integer $ n$, 
let $ m > 2 \frac {\ln n} {\delta ^2 }$.  For any set $ \mathbf X\subset \mathbb S ^{N-1}$ of cardinality $ n$, 
there is a $ \delta $-RIP from $ \mathbf X$ into $ \mathbb H_m$.  
The counter-intuitive fact is that our bound for $ m$ is \emph{virtually identical to the one that holds 
for the linear J-L Lemma. } 

We prove the One Bit J-L Lemma in \S \ref{s:JL}.  The simplier property of our random one-bit map being one to one is studied in \S \ref{s:11}.  
For special choices of $ \mathbf X$ , we make a finer analysis of the one-to-one and RIP properties.  
They satisfy phase transitions that depend only weakly on the number of points we are mapping, 
see \S \ref{s:PhaseRIP} and \S \ref{s:Phase11}. 
Some background information is recalled in \S \ref{s:background}.

\section{Background} \label{s:background} 
We formalize below several definitions we will use throughout the paper.
\subsection* {Hamming Cube}
 $\mathbb{H}_m =\{0,1\}^m$ for all $x\in\mathbb{H}_m$ $x=x_1x_2...x_n$ where $x_i\in \{0,1\}$. For all $x,y\in\mathbb{H}_m$ we have the normalized metric 
 $$d_{\mathbb{H}_m}(x,y)=\frac{1}{m}\#\{i:x_i\neq y_i\} . $$ 

 \subsection* {Random $m$-dimensional One-Bit Map}
Given $\{\theta_j \;:\; 1\leq j \leq m\}$ be iid uniformly distributed  random vectors in $\mathbb{S}^{N-1}$. 
Define a map $ B _m \;:\; \mathbb S ^{N-1} \to \mathbb H _m$ by  
$\textbf{B}_m x = \{ \textup{sgn}(x\cdot\theta_j) \} _{j=1} ^{m}$. 
Observe that 
\begin{equation}\label{e:distance}
d_{\mathbb{H}_m}(\textbf{B}_m x,  \textbf{B}_m y)
=\frac{1}{m}\sum_{i=1}^m\frac{|\textup{sgn}(x\cdot\theta_i)-\textup{sgn}(y\cdot\theta_i)|}{2}.
\end{equation}

\subsection*{Geodesic Distance} Fix $x$, $y$ on $\mathbb{S}^{N-1}$. The geodesic distance $d_{geo}(x,y),$  is the shortest distance between the points $x$ and $y$ on the surface. 
 This is given by 
 $$d_{geo}(x,y)=\frac{\cos^{-1}(x\cdot y)}{\pi}.$$ 
 Antipodal points are normalized to be distance one apart. Geodesic distance has this probabilistic interpretation: Let ${\textup{Wedge}}_{xy}=\{\theta\in\mathbb{S}^{N-1}:\textup{sgn}(x\cdot\theta)\neq\textup{sgn}(y\cdot\theta)\}.$ These are the $\theta$ which distinguish between $x$ and $y$ under the one-bit map. Selecting $\theta\in\mathbb{S}^{N-1}$ at random, the probability of being in ${\textup{Wedge}}_{xy}$ is $d_{geo}(x,y).$ This is an instance of the Crofton formula.

 For the distance in \eqref{e:distance}, we then have 
 \begin{gather}\label{e:Distance}
d_{\mathbb{H}_m}(\textbf{B}_m x,  \textbf{B}_m y)
=\frac{1}{m}\sum_{i=1}^m \mathbf 1_{{\textup{Wedge}}_{xy}} (\theta _j) . 
\end{gather}
The right hand side is an average of Bernoulli rvs.  In particular, the difference between the Hamming and 
geodesic metrics is 
\begin{equation} \label{e:difference}
d_{\mathbb{H}_m}(\textbf{B}_m x,  \textbf{B}_m y) - d _{\textup{geo}} (x,y)
=\frac{1}{m}\sum_{i=1}^m \mathbf 1_{{\textup{Wedge}}_{xy}} (\theta _j) 
- d _{\textup{geo}} (x,y) . 
\end{equation}
Standard deviation inequalities for Bernoulli rvs apply to the right hand side above. 
 
 \subsection*{The Restricted Isometry Property} $\textbf{B}_m : \textbf{X} \to \mathbb H _m$ has the $\delta$-RIP if for all pairs $x,y\in\textbf{X}$: $$|d_{\mathbb{H}_m}(\textbf{B}_m x,\textbf{B}_m y)-d_{geo}(x,y)|\leq\delta.$$

\subsection* {Positively Associated Stein-Chen Approximation} For random variables to be positively associated, their covariance is positive, meaning they increase or decrease together.

$$d_{TV}(W, \textup{Poi}(\lambda))\leq min\left(1,\frac{1}{\lambda}\right)\left(\textup{Var}(W)-\lambda+2\sum_{i\in I}P_i^2\right)$$ where $W$ is a sum of positively associated Bernoullis with parameter $P_i$, $\lambda$ is $\mathbb{E}(W)$, and $d_{TV}$ is total variation distance.
\subsection* {General Form of Stein-Chen Approximation}\cite{arratia1990poisson}
 $$d_{TV}(W, \textup{Poi}(\mathbb{\lambda}))\leq\min\left(1,\dfrac{1}{\lambda}\right)\sum_{i\in I}\left(P_i^2+\sum_{j\in\mathscr{N}_i}\left(P_iP_j+\mathbb{E}(X_iX_j)\right)\right)$$
where $X_i$ are Bernoullis with parameter $P_i,$ $W$ is a sum of all $X_i$, $\lambda=\mathbb{E}(W)$, $\mathscr{N}_i$ is the set of random variables that depend on $X_i$, and $d_{TV}$ is total variation distance.

\section{A One-to-One Mapping From the Unit Sphere to the Hamming Cube} \label{s:11}
We start with an analysis of a simpler property of $\textbf{B}_m$ being one-to-one.


\begin{theorem}
Let $0<\delta$, $\epsilon<1$, and let $\textbf{X}\subset\mathbb{S}^{N-1}$ be a subset of n points with $d_{geo}(x,y)>1-\delta$, where $x\neq y\in\bf{X}.$ The random $m$-dimensional one-bit map $\textbf{B}_m$ : $\textbf{X}\to\mathbb{H}_m$ will be one-to-one with probability at least $1-\epsilon$ provided that $$m\geq \dfrac{\ln\frac{n^2}{2\epsilon}}{\ln\frac{1}{\delta}}.$$
In the special case when the points $x$ and $y$ are pairwise orthogonal, $d_{geo}(x,y) = \frac{1}{2}$, $$m\geq 2\log_2n+\log_2\frac{1}{2\epsilon}.$$
\end{theorem}
By the pigeonhole principle, m must be at least $\log_2n.$ Our result shows that if $m>2\log_2n,$ then the random $m$-dimensional one-bit map is one-to-one with high probability. 
\begin{proof}By the union bound, we know that:$$\mathbb{P}(\textbf{B}_m \textrm{ is not one-to-one})\leq \sum_{x\backsim y}\mathbb{P}(\textbf{B}_mx=\textbf{B}_m y \textrm{ with } x\neq y) \leq \binom{n}{2}\delta^{m}.$$ In this expression, $\displaystyle\sum_{x\backsim y}$ means the sum over all unordered pairs $x\backsim y$ where $x\neq y$. Above, there are $\binom{n}{2}$ pairs $x\backsim y\in\bf{X}.$ The i\textsuperscript{th} coordinates of $\textbf{B}_mx$ and $\textbf{B}_m y$ are equal with probability  at most $\delta$. The coordinates are independent, hence the inequality above. We require $\binom{n}{2}\cdot\delta^m \leq\epsilon$, which is true if
$$m\geq \dfrac{\ln\frac{n^2}{2\epsilon}}{\ln\frac{1}{\delta}}.$$
This condition is sufficient for $\textbf{B}_m$ to be one-to-one with probability $1-\epsilon$.
In the special case when $\bf{X}$ consists of pairwise orthogonal vectors, $d_{geo}(x,y) = \frac{1}{2}$, the bound is
$$m \geq 2\log_2{n}+\log_2\frac{1}{2\epsilon}.$$
\end{proof}

\section{A Phase Transition in One-to-One Property} \label{s:Phase11}


For a special class of $ \textbf{X}$, we analyze the property of $ \textbf{B}_m$ 
being one-to-one. We show that the probability passes through a phase transition. And the width of the phase transition is essentially independent of the cardinality of $ \textbf{X}$.
\begin{theorem}
 Fix $0<\epsilon_2<\epsilon_1<1$. Let $\bf{X}$ be $n$ pairwise orthogonal vectors in $\mathbb{S}^{N-1}$, and let $P_{\textrm{1-1}}(m)$ be the probability that $\textbf{B}_m$ is one-to-one. Then for $n\geq10$, $1-\epsilon_1<P_{\textrm{1-1}}(m)$ when: 
$$\log_2\frac{n(n-1)}{2\ln\frac{1}{1-\frac{\epsilon_1}{1.01}}}\leq m$$ and $P_{\textrm{1-1}}(m)<1-\epsilon_2$ when
$$m\leq\log_2\frac{n(n-1)}{2\ln\frac{1}{1-\frac{\epsilon_2}{0.99}}}.$$
Additionally, the phase transition is bounded as follows: $$P_{\textrm{1-1}}(m)=\mathbb{P}( \textbf{B}_m \textrm{ is one-to-one})\in[e^{-\frac{\binom{n}{2}}{2^{m}}}-\binom{n}{2}\cdot 2^{-2m},e^{-\frac{\binom{n}{2}}{2^{m}}}+\binom{n}{2}\cdot 2^{-2m}].$$
\end{theorem}

\noindent We will analyze this from the perspective of the birthday problem. To do this, we will count all pairs of points that $\textbf{B}_m$ map to the same point in the Hamming cube. Namely for $x,y\in\textbf{X}$, let
$$ \hspace{.1cm} W_{x\backsim y}= \begin{cases} 
1,& \textbf{B}_mx = \textbf{B}_my\\
0,& \textrm{otherwise.}
\end{cases}$$
All $W_{x\backsim y}$ are i.i.d with probability $p=\frac{1}{2^m}$ and $W = \displaystyle\sum_{x\backsim y}W_{x\backsim y}$ is a sum of positively associated Bernoulli random variables. 
By the Stein-Chen approximation, $W$ is close to a Poisson distribution, in total variation, denoted $ d _{TV}$ below.  We make this precise below:  
$$d_{TV}(W, \textrm{Poi}(\mathbb{\lambda}))\leq\eta$$ where 
\begin{align}
\eta=\min\left(1,\dfrac{1}{\lambda}\right) [\textrm{Var}(W)-\lambda+2\sum_{x\backsim y}p^{2}].\label{eq:3}
\end{align}
\begin{lemma}
We claim $$\textrm{Var}(W) = \lambda -\displaystyle\sum_{x\backsim y}p^2.$$
\end{lemma}

\noindent\begin{proof}

If the the $W_{x\backsim y}$ have the property $\mathbb{E}[W_{x\backsim y}W_{r\backsim s}]=\mathbb{E}[W_{x\backsim y}]\mathbb{E}[W_{r\backsim s}]$ then the $W_{x\backsim y}$ are pairwise independent. Assuming this, variance adds, and $Var(W)$ can be calculated: 

\begin{eqnarray*} 
\textrm{Var}(W)&=&\sum_{x\backsim y}\textrm{Var}(W_{x\backsim y})\\
&=&\binom{n}{2} p(1-p)\\
&=& \lambda -\displaystyle\sum_{x\backsim y}p^2.\\
\end{eqnarray*}
\noindent It remains to prove that $\mathbb{E}[W_{x\backsim y}W_{r\backsim s}]=\mathbb{E}[W_{x\backsim y}]\mathbb{E}[W_{r\backsim s}]$. It will be sufficient to show that $\mathbb{E}[W_{x\backsim y}W_{r\backsim s}]=p^2$. The only non-trivial case is when $x\backsim y$ and $r\backsim s$ share exactly one point. We will write this as $\mathbb{E}[W_{x\backsim y}W_{y\backsim s}]$ where y is the shared point.
$\mathbb{E}[W_{x\backsim y}W_{y\backsim s}]$ is equal to 
$\mathbb{P}(\textbf{B}_mx=\textbf{B}_my=\textbf{B}_ms)$ which means we have three distinct points on the sphere mapping to the same point on the Hamming cube. Thus $\mathbb{P}(\textbf{B}_mx=\textbf{B}_my=\textbf{B}_ms)=p^2$ giving us that $\mathbb{E}[W_{x\backsim y}W_{r\backsim s}]=p^2$.





\end{proof}

\begin{lemma} We claim that $\eta=\binom{n}{2}\cdot2^{-2m}$ where $\eta$ is $\eqref{eq:3}$ 

\end{lemma}
Using this, we can bound $1-\mathbb{P}(W\geq1)$ in the window $$1-\mathbb{P}(W\geq1)\in[e^{-\frac{\binom{n}{2}}{2^{m}}}-\eta,e^{-\frac{\binom{n}{2}}{2^{m}}}+\eta].$$

\noindent\begin{proof}
We can find an expression for $\eta$ using the variance of $W$:
\begin{eqnarray*}
\eta&=&\min\left(1,\dfrac{1}{\lambda}\right) \left[\textrm{Var}(W)-\lambda+2\sum_{x\backsim y}p^{2}\right]\\
&=&\binom{n}{2}\cdot2^{-2m}.
\end{eqnarray*}
We can bound $1-\mathbb{P}(W\geq1)$ which is equal to $P_{\textrm{1-1}}(m)$:
$$|\mathbb{P}(W\geq1)-\mathbb{P}(Poi(\lambda)\geq1)|\leq\eta$$
$$1-\mathbb{P}(W\geq1)\in[e^{-\frac{\binom{n}{2}}{2^{m}}}-\eta,e^{-\frac{\binom{n}{2}}{2^{m}}}+\eta].$$ 
\end{proof}
\noindent \textbf{Solving For m.}
\noindent  
Fix $0<\epsilon_2<\epsilon_1<1$, let $\textbf{X}$ be $n$ pairwise orthogonal vectors in $\mathbb{S}^{N-1}$, and let $P_{\textrm{1-1}}(m)$ be the probability that $\textbf{B}_m$ is one-to-one, then $1-\epsilon_1P_{\textrm{1-1}}(m)$ when: 
\begin{eqnarray*}
1-e^{-\frac{\binom{n}{2}}{2^{m}}}+\eta&\leq&\epsilon_1.\\
\end{eqnarray*}
and $P_{\textrm{1-1}}(m)<1-\epsilon_2$ when:
\begin{eqnarray*}
1-e^{-\frac{\binom{n}{2}}{2^{m}}}-\eta&\geq&\epsilon_2\\
\end{eqnarray*}
In order to ensure that $\eta$ is very small compared to the Poisson distribution, we want $\eta\leq0.01(1-e^{-\frac{\binom{n}{2}}{2^{m}}}).$ If we fix $n$ and choose $m$ such that $\frac{\binom{n}{2}}{2^{m}}\leq1$, observe the inequality   
$$
\frac{\frac{\binom{n}{2}}{2^{m}}}{2}\leq\frac{\binom{n}{2}}{2^{m}} -\frac{\frac{\binom{n}{2}^2}{2^{2m}}}{2}. 
$$
It is sufficient to bound $\eta$ as $$\eta\leq0.01\frac{\frac{\binom{n}{2}}{2^{m}}}{2}.$$
Manipulating this statement we get: $$\frac{\binom{n}{2}\cdot2^{-m}}{\binom{n}{2}}\leq0.005.$$ Since we already assumed that $\frac{\binom{n}{2}}{2^{m}}\leq1$, we can rewrite this inequality as $\frac{1}{\binom{n}{2}}\leq0.005$ and solve for $n$: $n\geq10.$ This means that if $n\geq10$, we have $\eta\leq0.01(1-e^{-\frac{\binom{n}{2}}{2^{m}}})$ which allows us to rewrite our inequalities and gain bounds on $m$:
$$\log_2\frac{n(n-1)}{2\ln\frac{1}{1-\frac{\epsilon_1}{1.01}}}\leq m
\leq\log_2\frac{n(n-1)}{2\ln\frac{1}{1-\frac{\epsilon_2}{0.99}}}.$$

\section{A Union Bound for the Restricted Isometry Property}  \label{s:JL}

This is the One Bit version of the Johnson-Lindenstrauss Lemma.  In particular the quantitive 
bound on $ m$ below is nearly identical to the best known bound in the linear Johnson-Lindenstrauss Lemma.

\begin{theorem}
Fix $0<\delta<\frac{1}{2}$ and $0<\epsilon<1$, let $\textbf{X}$ be $n$ pairwise orthogonal vectors in $\mathbb{S}^{N-1}$.  The random $m$ dimensional one-bit map, $\textbf{B}_m:\textbf{X}\to\mathbb{H}_m$, satisfies the $\delta$-RIP with probability at least $1-\epsilon$
when $$m\geq \frac{\ln\frac{n^2}{\epsilon}}{2\delta^2}.$$
\end{theorem}

\noindent In the Restricted Isometry Property (RIP), we want to preserve the pairwise distances between the points so that for all pairs $x,y\in\textbf{X}$ $$|d_{\mathbb{H}_m}(\textbf{B}_m x,\textbf{B}_m y)-d_{geo}(x,y)|\leq\delta.$$ To ensure that we satisfy the $\delta$-RIP with probability at least $1-\epsilon$, we have to be certain that the probability of failure, $\epsilon$, is small:
$$\mathbb{P}( \exists  x,y\in\textbf{X} :|d_{\mathbb{H}_m}(\textbf{B}_m x,\textbf{B}_m y)-d_{geo}(x,y)|\geq\delta)\leq\epsilon.$$ 
Using the Union bound, 
\begin{eqnarray}
\mathbb{P}( \exists  x,y\in\textbf{X}:|d_{\mathbb{H}_m}(\textbf{B}_m x,\textbf{B}_m y)-d_{geo}(x,y)|\geq\delta)\leq\sum_{x\backsim y}\mathbb{P}(|d_{\mathbb{H}_m}(\textbf{B}_m x,\textbf{B}_m y)-d_{geo}(x,y)|\geq\delta).
\end{eqnarray}
We will first analyze the probability for one pair $x\backsim y$,  namely:
\begin{eqnarray}
\mathbb{P}(|d_{\mathbb{H}_m}(\textbf{B}_m x,\textbf{B}_m y)-d_{geo}(x,y)|\geq\delta).
\end{eqnarray}
\begin{lemma}
We claim that $\mathbb{P}(|d_{\mathbb{H}_m}(\textbf{B}_m x,\textbf{B}_m y)-d_{geo}(x,y)|\geq\delta)\leq 2e^{-2\delta^2m}$ for all pairs $x\backsim y.$
\end{lemma}
\begin{proof}
The difference in the metrics is the average of the random variables: $$\textbf{X}_i(x,y) = \frac{|\textup{sgn}(x\cdot\theta_i)-\textup{sgn}(y\cdot\theta_i)|}{2}-\frac{\cos^{-1}(x\cdot y)}{\pi}.$$ These are independent centered Bernoulli random variables that satisfy a large deviation inequality which is uniform in the Bernoulli parameter \cite{hoeffding1963probability},
$$\mathbb{P}\left(\left|\frac{1}{m}\sum_{i=1}^{m}\textbf{X}_i(x,y)\right|>\delta\right)\leq2e^{-2\delta^2m}.$$
\end{proof}

\noindent\textbf{The Union Bound for the RIP:}
The last expression above provides a bound for the probability that one pair $x\backsim y$  fails the $\delta$-RIP. Now, summing over all pairs,$$\sum_{x\backsim y}\mathbb{P}\left(\left|\frac{1}{m}\sum_{i=1}^n\frac{|\textup{sgn}(x\cdot\theta_i)-\textup{sgn}(y\cdot\theta_i)|}{2}-\frac{\cos^{-1}(x\cdot y)}{\pi}\right|\geq\delta\right)\leq2\binom{n}{2}\cdot e^{-2\delta^2m}.$$
We can bound this probability with $\epsilon$ and solve for $m$:
\begin{eqnarray*}
m&\geq&\frac{\ln\frac{n^2}{\epsilon}}{2\delta^2}.
\end{eqnarray*} 
For all $m$ greater than this bound, $m$ must satisfy the $\delta$-RIP with probability at least $1-\epsilon.$

\section{A Phase Transition for the Restricted Isometry Property}  \label{s:PhaseRIP}
For special $ \textbf{X}$, we analyze the property of $ \textbf{B}_m$ is a $ \delta $-RIP. 
Again, the size of the window's dependence on $ \lvert  \textbf{X}\rvert $, is very weak. This time the dependence is 
in terms of $ \ln\ln\lvert  \textbf{X}\rvert $.  
\begin{theorem}
Fix $0<\delta<\frac{1}{2}$, fix $0<\epsilon_2<\epsilon_1<0.99$. Let $\bf{X}$ be $n$ pairwise orthogonal vectors in $\mathbb{S}^{N-1}$, and let $P_{\textrm{RIP}}(m)$ be the probability that $\textbf{B}_m$ satisfies $\delta-\textrm{RIP}$. If $n\geq800$, then $1-\epsilon_1<P_{\textrm{RIP}}(m)$ when: 
\begin{align} \label{e:mUpper} 
m \leq\frac{1}{\frac{1}{2}\ln(1-4\delta^2)+\delta\ln\frac{1+2\delta}{1-2\delta}}\left[\ln\frac{n(n-1)}{2\sqrt{2\pi}e^{\frac{1}{6}}\ln\frac{1}{1-\frac{\epsilon_1}{1.01}}}-\ln\ln\frac{n(n-1)}{2\sqrt{2\pi}e^{\frac{1}{6}}\ln\frac{1}{1-\frac{\epsilon_1}{1.01}}}\right]\\ 
\end{align}  
and $P_{\textrm{RIP}}(m)<1-\epsilon_2$ when 
\begin{align} \label{e:mLower}
m &\geq& \frac{1}{\frac{1}{2}\ln(1-4\delta^2)+\delta\ln\frac{1+2\delta}{1-2\delta}}\left[\ln\frac{n(n-1)e^{\frac{1}{12}}}{2\sqrt{2\pi}\ln\frac{1}{1-\frac{\epsilon_2}{0.99}}}+\ln\ln\frac{n(n-1)e^{\frac{1}{12}}}{2\sqrt{2\pi}\ln\frac{1}{1-\frac{\epsilon_2}{0.99}}}\right].
\end{align}
Additionally, the phase transition is bounded as follows:$$\mathbb{P}(\textbf{B}_m \textrm{ is a }\delta-\textrm{RIP})\in[e^{-\lambda_2^\delta}-\eta^\delta,e^{-\lambda_1^\delta}+\eta^\delta]$$ where $\lambda_1^\delta=\frac{n(n-1)}{2}\cdot\frac{e^{\frac{-1}{6}}}{\sqrt{2\pi m}}\cdot e^{m[\frac{-1}{2}\ln(1-4\delta^2)+\delta\ln\frac{1-2\delta}{1+2\delta}]}$, $\lambda_2^\delta=\frac{n(n-1)}{2}\cdot\frac{e^{\frac{1}{12}}\sqrt{m}}{\sqrt{2\pi}}\cdot e^{m[\frac{-1}{2}\ln(1-4\delta^2)+\delta\ln\frac{1-2\delta}{1+2\delta}]},$\\ and $\eta^\delta=\binom{n}{2}\left[(p^{\delta})^2+4(n-2)(p^{\delta})^2\right]$ where  $p^\delta\leq\frac{e^{\frac{1}{12}}\sqrt{m}}{\sqrt{2\pi}}\cdot e^{m[\frac{-1}{2}\ln(1-4\delta^2)+\delta\ln\frac{1-2\delta}{1+2\delta}]}.$
\end{theorem}

The graph below shows a simulation of the RIP property with $ \delta = 0.2$. 
The red line is the bound \eqref{e:mUpper}, the green line is \eqref{e:mLower}.  The jagged blue line is the simulated value of the probability of $ \textbf{B}_m$ being a $ 0.2$-RIP.  The line is jagged, due to the discrete nature of the Hamming metric. The latter fact is of course a complication implicit in our proof.
\begin{figure}[H]
\centering
\includegraphics[width=0.7\linewidth]{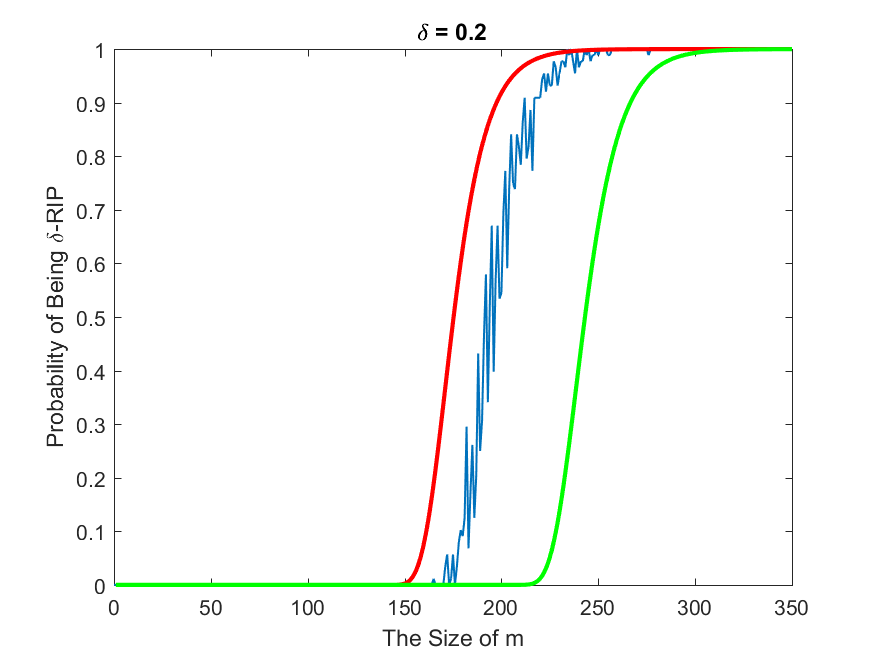}
\end{figure}

\noindent We will again analyze the phase transition from the perspective of the birthday problem. To do this we will count all $x\backsim y\in\textbf{X}$ that fail the RIP property, namely:$$\hspace{.05cm} W^\delta_{x\backsim y}= \begin{cases} 
1,& \textbf{B}_mx=\textbf{B}_my\\
0,& \textrm{otherwise} 
\end{cases}.$$  
Then $W^\delta =\displaystyle\sum_{x\backsim y}W^\delta_{x\backsim y}$ is a sum of Bernoulli random variables. All $W^\delta_{x\backsim y}$ are i.i.d with probability $p^\delta=\mathbb{P}(|d_{\mathbb{H}_m}(\textbf{B}_mx,\textbf{B}_my)-d_{geo}(x,y)|\geq\delta)$. 
By the general form of the Stein-Chen approximation $W^\delta$ is close to a Poisson distribution in total variation. We make this precise below:
$$d_{TV}(W^\delta, \textrm{Poi}(\lambda^\delta)) \leq \eta^\delta$$ where 
\begin{align}
\eta^\delta=min\left(1,\frac{1}{\lambda^\delta}\right)\sum_{x\backsim y}\left[(p^\delta_{x\backsim y})^2+\sum_{r\backsim s\in\mathscr{N}_{x\backsim y}}(p^\delta_{x\backsim y})^2+\mathbb{E}[W^\delta_{x\backsim y},W^\delta_{r\backsim s}]\right]\label{eq:2}\\
\end{align}
\begin{lemma}
We claim that $\lambda^\delta = \binom{n}{2}\mathbb{P}(|Y-\frac{m}{2}|>m\delta)$ where $Y$ is $\textrm{Bin}(m,\frac{1}{2})$.
\end{lemma}
\begin{proof}
We know that $\lambda^\delta=\binom{n}{2}p^\delta.$ In this special case, the geodesic distances between the points in $\bf{X}$ is $\frac{1}{2},$ which reduces $p^\delta$ to:
\begin{eqnarray*}
p^\delta&=&\mathbb{P}\left(\left|d_{\mathbb{H}_m}(\textbf{B}_mx,\textbf{B}_my)-\frac{1}{2}\right|\geq\delta\right)\\
&=&\mathbb{P}\left(\left|\sum_{i=1}^m\frac{|\textup{sgn}(x\cdot\theta_i)-\textup{sgn}(y\cdot\theta_i)|}{2}-\frac{m}{2}\right|\geq m\delta\right)\\
\end{eqnarray*}
For each $i$, $\frac{|\textup{sgn}(x\cdot\theta_i)-\textup{sgn}(y\cdot\theta_i)|}{2}$ is Bernoulli with parameter $d_{geo}(x,y)=\frac{1}{2}.$ The $\theta_i$ are independently so $Y=\sum_{i=1}^m\frac{|\textup{sgn}(x\cdot\theta_i)-\textup{sgn}(y\cdot\theta_i)|}{2}$ is $\textrm{Bin}(m,\frac{1}{2})$ and we can rewrite $p^\delta$ as: $$p^\delta=\mathbb{P}\left(\left|Y-\frac{m}{2}\right|\geq m\delta\right)$$ 
\end{proof}

\begin{lemma}  We can bound $\lambda^\delta$ as such:
$$\frac{e^{-\frac{1}{6}}}{\sqrt{m}} \leq \frac{\lambda^{\delta}}{\Lambda} \leq e^{\frac{1}{12}}\sqrt{m}$$ where $\Lambda=\frac{\binom{n}{2}}{\sqrt{2\pi}}\cdot e^{m[\frac{-1}{2}\ln(1-4\delta^2)+\delta\ln\frac{1-2\delta}{1+2\delta}]}$ Additionally for $\delta<0.25$, we can approximate this statement as:$$\binom{n}{2}\cdot\frac{e^{-\frac{1}{6}}}{\sqrt{2\pi m}}\cdot e^{\frac{-2m\delta^2 -4m\delta^3}{1-2\delta}} \hspace{.5cm} \leq \lambda^\delta \leq \binom{n}{2}\cdot\frac{e^{\frac{1}{12}}\sqrt{m}}{\sqrt{2\pi}}\cdot e^{\frac{-2m\delta^2 +8m\delta^3}{1-4\delta^2}}$$
\end{lemma}
\begin{proof}
As previously defined, 
\begin{align}
\lambda^\delta &= \binom{n}{2}\mathbb{P}(|Y-\frac{m}{2}|>m\delta)\\
&=2\binom{n}{2}\frac{1}{2^m}\left[\binom{m}{\lceil\frac{m}{2}+m\delta\rceil}+...+\binom{m}{m}\right]\label{eq:1}.
\end{align}

\noindent Now we will use Sterling's approximation, 
$$\sqrt{2\pi m}\left(\frac{m}{e}\right)^{m}\leq m!\leq\sqrt{2\pi m}\left(\frac{m}{e}\right)^{m}\cdot e^{\frac{1}{12m}},$$
to obtain bounds for $\lambda^{\delta}$, but for the upper bound, we will use the fact that $\displaystyle{e^{\frac{1}{12m}}\leq e^{\frac{1}{12}}}$ for $m\geq1$.\\

\noindent Let $A=\lceil\frac{m}{2}+m\delta\rceil$ and let us assess only the first term of the $\lambda^{\delta}$ sum since it is the largest.

\begin{align*}
e^{-\frac{1}{6}}\cdot\sqrt{\frac{m}{2\pi A(m-A)}}\cdot\frac{\left(\frac{m}{e}\right)^{m}}{\left(\frac{A}{e}\right)^{A}\left(\frac{m-A}{e}\right)^{m-A}}
&\leq
\binom{m}{A}
&\leq&
\sqrt{\frac{m}{2\pi A(m-A)}}\cdot\frac{\left(\frac{m}{e}\right)^{m}}{\left(\frac{A}{e}\right)^{A}\left(\frac{m-A}{e}\right)^{m-A}}\cdot e^{\frac{1}{12}}\\
\frac{\frac{e^{-\frac{1}{6}}m^m}{\sqrt{2\pi m}}}{\left(\frac{m^{2}}{4}-m^{2}\delta^{2}\right)^{\frac{m}{2}}\left(\frac{\frac{m}{2}+m\delta}{\frac{m}{2}-m\delta}\right)^{m\delta}}
&\leq 
\binom{m}{A} 
&\leq&
\frac{\frac{e^{\frac{1}{12}}m^m}{\sqrt{2\pi m}}}{\left(\frac{m^{2}}{4}-m^{2}\delta^{2}\right)^{\frac{m}{2}}\left(\frac{\frac{m}{2}+m\delta}{\frac{m}{2}-m\delta}\right)^{m\delta}}\\
\frac{\frac{e^{-\frac{1}{6}}}{\sqrt{2\pi m}}\left(1-2\delta\right)^{m\delta}}{\left(\frac{1}{4}-\delta^{2}\right)^{\frac{m}{2}}\left(1+2\delta\right)^{m\delta}}
&\leq
\binom{m}{A}
&\leq&
\frac{\frac{e^{\frac{1}{12}}}{\sqrt{2\pi m}}\left(1-2\delta\right)^{m\delta}}{\left(\frac{1}{4}-\delta^{2}\right)^{\frac{m}{2}}\left(1+2\delta\right)^{m\delta}}.
\end{align*}
We can rewrite these bounds in terms of $\Lambda$:
$$\frac{2^me^{-\frac{1}{6}}}{\binom{n}{2}\sqrt{m}} \leq \frac{\binom{m}{A}}{\Lambda} \leq \frac{2^me^{\frac{1}{12}}\sqrt{m}}{\binom{n}{2}}$$
Using this inequality for the first term in $\eqref{eq:1}$, we gain a lower bound for $\lambda^{\delta}$:  
$$\frac{\Lambda e^{-\frac{1}{6}}}{\sqrt{m}} \leq \lambda^{\delta}.$$
For the upper bound, we have at most $m$ summands in $\eqref{eq:1}$ so the upper bound is: $$\lambda^{\delta} \leq \Lambda e^{\frac{1}{12}}\sqrt{m}$$
\noindent For $\delta<0.25$ we can simply these two statements above using the Taylor approximation for  $\ln(1-x)$, $\frac{-x}{1-x} \leq \ln(1-x) \leq -x$

\begin{align*}
\binom{n}{2}\cdot\frac{e^{-\frac{1}{6}}}{\sqrt{2\pi m}}\cdot e^{\frac{-2m\delta^2 -4m\delta^3}{1-2\delta}} \hspace{.5cm} \leq \lambda^\delta \leq \binom{n}{2}\cdot\frac{e^{\frac{1}{12}}\sqrt{m}}{\sqrt{2\pi}}\cdot e^{\frac{-2m\delta^2 +8m\delta^3}{1-4\delta^2}}.
\end{align*}
\end{proof}

\begin{lemma}
We claim that: $\eta^\delta=\binom{n}{2}\left[(p^{\delta})^2+4(n-2)(p^{\delta})^2\right]$ where $\eta^\delta$ is given in $\eqref{eq:2}.$
Using this we can bound $1-\mathbb{P}(W^\delta\geq1)$ in the window $$[e^{-\lambda_2^\delta}-\eta^\delta,e^{-\lambda_1^\delta}+\eta^\delta]$$ where $\lambda_1^\delta=\binom{n}{2}\cdot\frac{e^{-\frac{1}{6}}}{\sqrt{2\pi m}}\cdot e^{m[\frac{-1}{2}\ln(1-4\delta^2)+\delta\ln\frac{1-2\delta}{1+2\delta}]}$ and $\lambda_2^\delta=\binom{n}{2}\cdot\frac{e^{\frac{1}{12}}\sqrt{m}}{\sqrt{2\pi}}\cdot e^{m[\frac{-1}{2}\ln(1-4\delta^2)+\delta\ln\frac{1-2\delta}{1+2\delta}]}.$\\
\end{lemma}
\begin{proof}
We recall: $$\eta^\delta= \min\left(1,\frac{1}{\lambda^\delta}\right)\sum_{x\backsim y}\left[(p^\delta)^2+\sum_{r\backsim s\in\mathscr{N}_{x\backsim y}}((p^\delta)^2+\mathbb{E}[W^\delta_{x\backsim y},W^\delta_{r\backsim s}])\right].$$ In order to estimate $\eta^\delta$, we need to estimate $|\mathscr{N}_{x\backsim{y}}|$. Because the only coordinates that are dependent on $x\backsim y$ are those that share exactly one coordinate with $x\backsim y$, $|\mathscr{N}_{x\backsim{y}}|\leq2(n-2)$. There are two ways this can happen: either $r$ or $s$ shares a coordinate with $x\backsim y.$ There are $n-2$ ways to choose the remaining coordinates. Assuming pairwise independence, we can estimate the size of $\eta^\delta$:

\begin{eqnarray*}
\eta^\delta\leq \binom{n}{2}\left[(p^{\delta})^2+4(n-2)(p^{\delta})^2\right].
\end{eqnarray*}
We can now bound $1-\mathbb{P}(W^\delta\geq1)$ which is equal to $\mathbb{P}(\textbf{B}_m \textrm{ is }\delta-\textrm{RIP})$. $$|\mathbb{P}(W^\delta\geq1)- \mathbb{P}(\textrm{Poi}(\lambda^\delta)\geq1)|<\eta^\delta$$ $$1-\mathbb{P}(W^\delta\geq1)\in[e^{-\lambda_2^\delta}-\eta^\delta,e^{-\lambda_1^\delta}+\eta^\delta]$$ where $\lambda_1^\delta=\binom{n}{2}\cdot\frac{e^{-\frac{1}{6}}}{\sqrt{2\pi m}}\cdot e^{m[\frac{-1}{2}\ln(1-4\delta^2)+\delta\ln\frac{1-2\delta}{1+2\delta}]}$ and $\lambda_2^\delta=\binom{n}{2}\cdot\frac{e^{\frac{1}{12}}\sqrt{m}}{\sqrt{2\pi}}\cdot e^{m[\frac{-1}{2}\ln(1-4\delta^2)+\delta\ln\frac{1-2\delta}{1+2\delta}]}.$

\end{proof}
\begin{lemma}
$W^{\delta}_{x\backsim y}$ are pairwise independent for all pairs, $x\backsim y$.
\end{lemma}
\begin{proof}
To show that $W^\delta_{x\backsim y}$ are pairwise independent, it is sufficient to show that $\mathbb{E}(W^\delta_{x\backsim y},W^\delta_{r\backsim s})=(p^\delta)^2:$  $$\mathbb{E}(W^\delta_{x\backsim y},W^\delta_{r\backsim s})= \mathbb{P}(W^\delta_{x\backsim y}=W^\delta_{r\backsim s}=1)=\mathbb{P}(W^\delta_{x\backsim y}=W^\delta_{y\backsim s}=1).$$ The only non-trivial case is when $x\backsim y$ and $r\backsim s$ share a common point. We can rewrite this probability as: $$\sum_{k\in\mathbb{H}_m}\mathbb{P}\left(\left|d_{\mathbb{H}_m}(\textbf{B}_mx,k)-\tfrac{1}{2}\right|\geq \delta, \left|d_{\mathbb{H}_m}(\textbf{B}_ms,k)-\tfrac{1}{2}\right|\geq\delta, \textbf{B}_my=k\right)$$ where k is an element of $\mathbb{H}_m.$ After an orthogonal transformation, we can take $x,y,s$ to be the first three coordinates vectors $e_1,e_2,e_3.$ The distribution of the $\theta_j$ are unchanged. The signs of the coordinates of the $\theta_j$ are independent, so the events are independent. Because all of these events are independent, the probability can be written as: 
\begin{eqnarray*}
\sum_{k\in\mathbb{H}_m}\mathbb{P}\left(\left|d_{\mathbb{H}_m}(\textbf{B}_mx,k)-\frac{1}{2}\right|\geq\delta\right)\mathbb{P}\left(\left|d_{\mathbb{H}_m}(\textbf{B}_ms,k)-\frac{1}{2}\right|\geq\delta\right)\mathbb{P}(\textbf{B}_my=k)\\ =\sum_{k\in\mathbb{H}_m}(p^\delta)^2\cdot2^{-m}.
\end{eqnarray*}
Because each  of these probabilities is identical distributed to $\mathbb{P}\left(\left|d_{\mathbb{H}_m}(\textbf{B}_mx,\textbf{B}_my)-\frac{1}{2}\right|\geq\delta\right).$ There are are $2^m$ elements in $\mathbb{H}_m$, we get that $\mathbb{P}(W^\delta_{x\backsim y}=W^\delta_{r\backsim s}=1) = (p^\delta)^2$, as desired.
\end{proof}
\noindent \textbf{Solving For $\bf{m}$}:
Using the previous bounds on $\lambda^\delta$, $$\lambda_1^\delta<\lambda^\delta<\lambda_2^\delta$$ where $\lambda_1^\delta=\binom{n}{2}\cdot\frac{e^{-\frac{1}{6}}}{\sqrt{2\pi m}}\cdot e^{m[\frac{-1}{2}\ln(1-4\delta^2)+\delta\ln\frac{1-2\delta}{1+2\delta}]}$ and $\lambda_2^\delta=\binom{n}{2}\cdot\frac{e^{\frac{1}{12}}\sqrt{m}}{\sqrt{2\pi}}\cdot e^{m[\frac{-1}{2}\ln(1-4\delta^2)+\delta\ln\frac{1-2\delta}{1+2\delta}]}.$
Fix $0<\epsilon_2<\epsilon_1<1$, let $\textbf{X}$ be $n$ pairwise orthogonal vectors in $\mathbb{S}^{N-1}$, and let $P_{\textrm{RIP}}(m)$ be the probability that $\textbf{B}_m$ is one-to-one, then $1-\epsilon_1<P_{RIP}(m)$ when:
\begin{eqnarray*}
1-e^{-\lambda_1^\delta}+\eta^\delta&\leq&\epsilon_1.\\
\end{eqnarray*} 
and $P_{\textrm{RIP}(m)}\leq1-\epsilon_2$ when 
\begin{eqnarray*}
1-e^{-\lambda_2^\delta}-\eta^\delta&\geq&\epsilon_2.\\
\end{eqnarray*}
In order to ensure that $\eta^\delta$ is very small compared to $1-e^{-\lambda_2^\delta}$ and $1-e^{-\lambda_1^\delta}$, we want $\eta^\delta\leq0.01(1-e^{-\lambda_2^\delta}).$ If we fix $n$ and choose $m$ such that $\lambda_2^\delta\leq1,$ then using the inequality $\frac{\lambda_2^\delta}{2}\leq\lambda_2^\delta -\frac{(\lambda_2^\delta)^2}{2}$ it is sufficient to bound $\eta^\delta$ as $$\eta^\delta\leq0.01\frac{\lambda_2^\delta}{2}.$$
Manipulating this statement we gain: $$(1+2(n-2))\cdot\frac{e^{\frac{1}{12}}\sqrt{m}}{\sqrt{2\pi}}\cdot  e^{m[\frac{-1}{2}\ln(1-4\delta^2)+\delta\ln\frac{1-2\delta}{1+2\delta}]}\leq0.005.$$ Because $1+2(n-2)\leq4\frac{n(n-1)}{2}$, we can rewrtite this inequality as: $$\frac{4\lambda_2^\delta}{n}\leq0.005.$$ Since we assumed that $\lambda_2^\delta\leq1$, we can rewrite this inequality as $\frac{4}{n}\leq0.005$ and solve for $n$: $$n\geq800.$$ This means that if $n\geq800$, we have $\eta^\delta\leq0.01(1-e^{-\lambda_2^\delta})$ and $\eta^\delta\leq0.01(1-e^{-\lambda_1^\delta})$ which allows us to rewrite our inequalities and get bounds on $m$:

\begin{eqnarray*}
1.01(1-e^{-\lambda_1^\delta})&\leq&\epsilon_1\\
\frac{1}{2}\ln m + m\left(\frac{1}{2}\ln(1-4\delta^2)+\delta\ln\frac{1+2\delta}{1-2\delta}\right) &\leq& \ln\frac{n(n-1)}{2\sqrt{2\pi}e^{\frac{1}{6}}\ln\frac{1}{1-\frac{\epsilon_1}{1.01}}}\\  
\end{eqnarray*} 
\begin{eqnarray*}
m \leq \frac{1}{\frac{1}{2}\ln(1-4\delta^2)+\delta\ln\frac{1+2\delta}{1-2\delta}}\left[\ln\frac{n(n-1)}{2\sqrt{2\pi}e^{\frac{1}{6}}\ln\frac{1}{1-\frac{\epsilon_1}{1.01}}}-\ln\ln\frac{n(n-1)}{2\sqrt{2\pi}e^{\frac{1}{6}}\ln\frac{1}{1-\frac{\epsilon_1}{1.01}}}\right]
\end{eqnarray*}
and

\begin{eqnarray*}
0.99(1-e^{-\lambda_2^\delta}) &\geq& \epsilon_2\\
m\left(\frac{1}{2}\ln(1-4\delta^2)+\delta\ln\frac{1+2\delta}{1-2\delta}\right) - \frac{1}{2}\ln m &\geq& \ln\frac{n(n-1)e^{\frac{1}{12}}}{2\sqrt{2\pi}\ln\frac{1}{1-\frac{\epsilon_2}{0.99}}}\\
\end{eqnarray*}
\begin{eqnarray*}
m \geq \frac{1}{\frac{1}{2}\ln(1-4\delta^2)+\delta\ln\frac{1+2\delta}{1-2\delta}}\left[\ln\frac{n(n-1)e^{\frac{1}{12}}}{2\sqrt{2\pi}\ln\frac{1}{1-\frac{\epsilon_2}{0.99}}}+\ln\ln\frac{n(n-1)e^{\frac{1}{12}}}{2\sqrt{2\pi}\ln\frac{1}{1-\frac{\epsilon_2}{0.99}}}\right].
\end{eqnarray*}

Let $q=\frac{1}{\frac{1}{2}\ln(1-4\delta^2)+\delta\ln\frac{1+2\delta}{1-2\delta}}.$ We remark that $q$ is approximately $\frac{1}{2\delta^2}.$
Let $r=\ln\frac{n(n-1)}{2\sqrt{2\pi}} $ then:
$$m\leq q\left[r+\ln\frac{1}{e^{\frac{1}{6}}\ln\frac{1}{1-\frac{\epsilon_1}{1.01}}} -\ln r-\ln\ln\frac{1}{e^{\frac{1}{6}}\ln\frac{1}{1-\frac{\epsilon_1}{1.01}}}\right]$$
and
$$m\geq q\left[r+\ln\frac{e^{\frac{1}{12}}}{\ln\frac{1}{1-\frac{\epsilon_2}{0.99}}} +\ln r+\ln\ln\frac{e^{\frac{1}{12}}}{\ln\frac{1}{1-\frac{\epsilon_2}{0.99}}}\right]$$
This is a statement of the main theorem by inspection.

\section{Acknowledgments} 
We would like to thank Dr. Michael Lacey and Dr. Robert Kesler for their assistance and mentorship. We would also like to thank the Georgia Institute of Technology and the NSF for their funding and support.

\bibliographystyle{alpha,amsplain}


\end{document}